\documentclass[letterpaper,10pt,reqno]{amsart}
\usepackage{indentfirst} 
\usepackage{amssymb}
\usepackage{mathrsfs}
\usepackage{graphicx} 
\usepackage{amsmath}
\usepackage{amsthm}
\usepackage{amssymb}
\usepackage{amscd}
\usepackage{amsbsy}
\usepackage{epsfig}
\usepackage{graphicx}
\usepackage{psfrag}
\usepackage{bbm}
\usepackage{hyperref}
\usepackage[UKenglish]{babel}
\usepackage[UKenglish]{isodate}
\usepackage{enumerate}

\theoremstyle{plain}
\newtheorem{theorem}{Theorem}[section]
\newtheorem{remark}[theorem]{Remark}
\newtheorem{proposition}[theorem]{Proposition}
\newtheorem{lemma}[theorem]{Lemma}

\newtheorem{definition}[theorem]{Definition}

\newtheorem*{lemma*}{Lemma}

\theoremstyle{definition}

\newtheorem{example}{Example}[section]

\def\R{\ensuremath{\mathbb R}}
\def\N{\ensuremath{\mathbb N}}

\def\ie{{\em i.e.}, }

\def\dim{\ensuremath{\text{dim}}}

\def\eps{\varepsilon}

\numberwithin{equation}{section}

\usepackage{graphicx} 
\usepackage[margin=30mm]{geometry}

\begin{document}

\title{Convergence of the Birkhoff spectrum for nonintegrable observables}

\date{\today}

\begin{thanks}
{G.I.\ was partially supported  by Proyecto Fondecyt 1230100.  MT was partially supported by the FCT (Funda\c c\~ao para a Ci\^encia  e a Tecnologia) project 2022.07167.PTDC. BZ is supported by a grant from Chinese Scholarship Council Programme.}
\end{thanks}

\author[G.~Iommi]{Godofredo Iommi}
\address{Facultad de Matem\'aticas,
Pontificia Universidad Cat\'olica de Chile (UC), Avenida Vicu\~na Mackenna 4860, Santiago, Chile}
\email{\href{godofredo.iommi@gmail.com}{godofredo.iommi@gmail.com}}
\urladdr{\url{http://http://www.mat.uc.cl/~giommi/}}

\author[M.~Todd]{Mike Todd}
\address{Mathematical Institute\\
University of St Andrews\\
North Haugh\\
St Andrews\\
KY16 9SS\\
Scotland} \email{\href{mailto:m.todd@st-andrews.ac.uk}{m.todd@st-andrews.ac.uk}}
\urladdr{\url{https://mtoddm.github.io/}}

\author[B.~Zhao]{Boyuan Zhao}
\address{\footnote{This work was done during BZ's PhD studies at the University of St Andrews. His current contact address is BICMR, Peking University, No. 5 Yiheyuan Rd, Haidian District, Beijing, P.R.China 100871}Mathematical Institute\\
University of St Andrews\\
North Haugh\\
St Andrews\\
KY16 9SS\\
Scotland} \email{\href{mailto:boyuan.zhao@pku.edu.cn}{boyuan.zhao@pku.edu.cn}}
\urladdr{\url{https://boyuanzhao.github.io/}}

\begin{abstract}
 We consider interval maps with countably many full branches and observables with polynomial tails. We show that the Birkhoff spectrum is real analytic and that its convergence to the Hausdorff dimension of the repeller is governed by the polynomial tail exponent. This result extends previous work by Arima on more regular observables and demonstrates how the tail behaviour influences the structure of the Birkhoff spectrum. Our proof relies on techniques from thermodynamic formalism and tail estimates for the observable and our applications are to natural observations on Gauss maps, L\"uroth transformations as well as to a first return time for a class of induced Manneville-Pomeau maps.
\end{abstract}

\subjclass[2020]{37C45, 37D35, 37E05}


\maketitle

\section{Introduction}

Averaging observables along an orbit is a basic procedure in ergodic theory.  The pointwise ergodic theorem ensures that, with respect to any ergodic measure, these averages converge to a constant almost everywhere.  However, it does not provide information about the exceptional set of measure zero.  For large classes of dynamical systems $T:X \to X$ and functions $g: X \to \R$,
the decomposition induced by its Birkhoff averages exhibits considerable complexity. Indeed, for $\alpha \in \R$ the associated level set is defined by
\begin{equation*}
J(\alpha)= \left\{x \in X : \lim_{n \to \infty} \frac{1}{n} \sum_{i=0}^{n-1} g(T^i x)= \alpha 		\right\}.
\end{equation*}
This induces a decomposition of  $X$ given by
\begin{equation*}
X= \bigcup_{\alpha \in \R} J(\alpha) \cup J_{ir},
\end{equation*}
where $J_{ir}$ is the set of points for which the Birkhoff average fails to converge. In many situations every level set $J(\alpha)$ is dense in $X$. When $X$  is a metric space, the structure of this decomposition can be analyzed via the function $\alpha \mapsto b(\alpha):= \dim_H J(\alpha)$, where $ \dim_H$ denotes the Hausdorff dimension. The function $b( \cdot )$ is referred to as the  \emph{Birkhoff spectrum}. The study of the Birkhoff spectrum began in the late 1990s in the context  of  hyperbolic conformal dynamical  systems  and H\"older functions  (see \cite{pes}). In this setting, if $g$ is not cohomologous to a constant, the decomposition induced by the Birkhoff averages is extremely complicated. 
The set of attainable values $\alpha$ forms a nontrivial interval, and each level set  $J(\alpha)$ is dense in $X$. Despite this complexity, the function $b$ is remarkably regular. Indeed, it is real analytic. This regularity is a consequence of the well-behaved thermodynamic formalism associated with the system.

We consider dynamical systems given by interval maps  $F:[0,1] \to [0,1]$ with countably many (full) branches. The dynamics is concentrated in the repeller $\Lambda \subset [0,1]$, and we assume the existence of a measure of maximal dimension $\mu$,  see Section \ref{sec:map} for details. We also consider  functions $\tau: \Lambda \to \R_{\ge 0}$, referred to as observables. These are assumed to satisfy conditions (see Section \ref{sec:obs}) which imply  that their  tail with respect to the measure $\mu$ has polynomial decay, meaning  there exists $0<\beta<1$ such that  $\mu(\tau>n)\asymp n^{-\beta}$, so $\int\tau~d\mu=\infty$\footnote{Since $\int\tau~d\mu= \sum_{n=0}^\infty\int_{\{n<\tau\le n+1\}}\tau~d\mu \ge \sum_{n=1}^\infty n\mu(\{n<\tau\le n+1\})\ge \sum_{n=1}^\infty\mu(\tau>n)\gtrsim \sum_{n=1}^\infty \frac1{n^\beta}=\infty$.}.  In this setting, we go on to prove that   the function $b(\alpha)$ is  real analytic and that
$\lim_{\alpha \to \infty} b(\alpha)=\dim_H \Lambda$, see Proposition \ref{lemma: q b are analytic}. Moreover, in analogy to recent results obtained by Arima  \cite{Ari2}, we prove the following,

\begin{theorem}\label{thm: main theorem}
Let $\Lambda \subset [0,1]$ be the repeller associated to the map $F$ (as in Section \ref{sec:map}) and let $\tau:\Lambda \to \R$ be an observable (as in Section \ref{sec:obs}), then
\begin{equation}
    \lim_{\alpha\to\infty} \left(\dim_H \Lambda -\dim_H J(\alpha)\right)\alpha^x=
 \begin{cases}
       0 &  \text{ if }  x<\frac{\beta}{1-\beta};\\
       \infty & \text{ if } x>\frac\beta{1-\beta}.
    \end{cases} 
\end{equation}    
\end{theorem}

In other words, the rate at which $b(\alpha)$ converges to $\dim_H \Lambda$ is controlled by a function of the polynomial exponent of the tail. This, in particular, implies that for observables as described in Section \ref{sec:obs}, different values of $\beta$ lead to fundamentally distinct Birkhoff spectra. 

As above, the observables we consider have (non-summable) polynomial tails with respect to the geometric measure as a consequence of more subtle (and technical) assumptions. These are hypotheses that allow us to control the growth and the interplay between  the number of cylinders of length one  for which the values of the observable lie in a certain fixed interval (Section~\ref{sec:obs} \hyperref[H1]{(H1)}--\hyperref[H3]{(H3)} ) and the $\mu$ measure of the these cylinders (Section \ref{sec:obs} \hyperref[H2]{(H2)}). In Section \ref{sec:simple}, under somewhat stronger assumptions,  we clarify the meaning and implications of  \hyperref[H1]{(H1)}  \hyperref[H2]{(H2)} and  \hyperref[H3]{(H3)}.

A natural example for this theory is the induced map of the Manneville-Pomeau map together with the observable defined as the corresponding first return time map (see Example~\ref{ex1}). In Section~\ref{sec:ex} we construct several examples that illustrate the assumptions on the observable and the system. Systems and observables that satisfy the hypotheses of our results 
can also occur as first return time maps for flows. This is readily seen simply by considering the suspension flow with base system $F$ and roof function equal to $\tau$. Other examples stem from the maps associated to certain numeration systems such as continued fractions expansions.

A result related to Theorem \ref{thm: main theorem} was obtained by Arima \cite{Ari2} for similar dynamical systems,  but for a class of more regular observables. In \cite[Theorem 1.1]{Ari2}   it is shown that the Birkhoff spectrum again converges to the Hausdorff dimension of the repeller at a polynomial rate. However, in that setting, the rate is determined by the Hausdorff dimension of the repeller, $\dim_H \Lambda$. Contrary to our Theorem, in order for that result to hold the dimension has to be strictly smaller than one.

It should be stressed that for this class of  dynamical systems but for a different class of observables,  the convergence  of the Birkhoff spectra to the Hausdorff dimension of the repeller can  be exponential \cite[Proposition 1.2]{Ari2} (see also \cite{IomJor15}).

Our method of proof combines modifications of ideas developed by Arima in \cite{Ari2} together with tail estimates for the observable. All of these are thermodynamic in nature.

\section{Preliminaries}
In this section we introduce the objects we are going to study throughout the article, we also recall results  and definitions  
that will be used and take the opportunity to fix notation. For two real sequences $(a_n)_n, (b_n)_n$, write $a_n=C^{\pm} b_n$ if there exists a uniform constant $C>0$ such that for all $n\in\N$, $C^{-1}\le\frac{a_n}{b_n}\le C$; write $a_n\lesssim b_n$ if for all $n$ large, $a_n\le C b_n$, and $a_n\asymp b_n$ if $a_n\lesssim b_n$ and $b_n\lesssim a_n$. The cardinality of a set $E$ is denoted by $\#E$.

\subsection{The symbolic space and thermodynamic formalism} 
Let $(\Sigma, \sigma)$  be the one-sided full-shift over the countable alphabet $\mathcal{A}=\N$. 
That is, we consider the space $\Sigma:= \left\{ (x_n)_{n \in \N} :  x_i \in \N\right\}$ together with   
the shift map $\sigma: \Sigma \to \Sigma$ defined by $\sigma(x_1x_2 x_3 \dots)=(x_2 x_3 \dots)$. 
Endow the space $\Sigma$ with the topology generated by the cylinder sets,
\[[i_1 i_2 \dots i_n]:=\left\{ (x_n) \in \Sigma : x_i=i_j \text{ for } j \in \{1,2, \dots n\}\right\}.\]
Note that with this topology the space $\Sigma$ is non-compact. Moreover, the system has infinite entropy. We define the \emph{$n$-th variation} of a function $\psi:\Sigma \to \R$ by
\[\text{var}_n(\psi)= \sup_{(i_1 \dots i_n) \in \N^n} \sup_{x,y \in [i_1 i_2 \dots i_n]} |\psi(x)-\psi(y)| . \]
A function $\psi:\Sigma \to \R$ is of \emph{summable variations} if $\sum_{j=2}^{\infty} \text{var}_j(\psi) <\infty$ and  
\emph{locally H\"older} if there exists $M>0$, $A_\psi\in(0,1)$ such that $\text{var}_j(\psi)\le MA_\psi^j$ for all $j\ge1.$

The thermodynamic formalism in this setting was studied by Mauldin and Urba\'nski \cite{MauUrb03} and also by Sarig \cite{Sar99}. We briefly recall  the basic properties that will be used through the article. Let $\psi: \Sigma \to \R$ be a function of summable variations. The \emph{topological pressure} of $\psi$ is defined by
 \[P_\sigma(\psi)=\lim_{n\to\infty}\frac1n\log\sum_{\sigma^nx=x}\exp \left(\sum_{i=0}^{n-1}\psi( \sigma^i x)\right).\]
Denote by $ \mathcal{M}_\sigma$ the space of $\sigma-$invariant probability measures.  The Variational Principle \cite{MauUrb03, Sar99} states that
  \[P_\sigma(\psi)=\sup\left\{h(\nu)+\int\psi\,d\nu:\nu\in \mathcal{M}_\sigma  , \int \psi\,d\nu>-\infty\right\},\]
where $h(\nu)$ denotes the entropy of the measure $\nu$.  An invariant measure $\nu \in \mathcal{M}_{\sigma}$ is an \emph{equilibrium state} for $\psi$ if $P_\sigma(\psi)=h(\nu)+\int\psi\,d\nu$, whenever $P_\sigma(\psi)<\infty$.

An equivalent condition of the pressure being finite is given by the following  result.
 \begin{lemma}\cite[Lemma 2.1.9]{MauUrb03}\label{lemma: MauUrb}
  We have that   $P_\sigma(\psi)<\infty$ if and only if 
     \[\sum_{a\in\mathcal A}\exp\left(\sup_{x\in[a]}\psi(x)\right)<\infty.\] 
 \end{lemma}
  In addition, it is possible to approximate the pressure defined on this non-compact space by the topological pressure of compact invariant subspaces.

\begin{lemma}(Approximation property, \cite[Theorem 2.1.6]{MauUrb03})\label{lemma: approximation property of pressure}
    If $\psi:\Sigma\to\R$ has summable variations then
    \[P_\sigma(\psi)=\sup\left\{h(\nu) + \int \psi\, d\nu:\nu  \in \mathcal{M}_{\sigma }\text{ and compactly supported}\right\}. \]
\end{lemma}
The above result also means that if $K \subset \Sigma$ is a compact invariant subset and we denote by $P_K(\cdot)$ the restriction of the pressure to the set $K$, then $P_\sigma(\psi) = \lim_{K}P_K(\psi)$ where the limit is taken over larger and larger invariant compact subsets. There is a special class of measures that will be of interest to us.

\begin{definition}
    A measure $\nu \in \mathcal{M}_{\sigma}$ is said to have the \emph{Gibbs property} if there exists $M>1$ such that for all $n\in\N$, all $n-$cylinder $[i_1,  \dots , i_n]$ and $x\in [i_1,  \dots , i_n]$,
    \begin{equation}\label{eqn: def Gibbs property}
        \frac1M\le \frac{\nu([i_1,  \dots , i_n])}{\exp\left(S_n\psi(x)-nP(\psi)\right)}\le M.
    \end{equation}
\end{definition}
The following result summarises the properties of the pressure in this setting, see \cite{MauUrb03,Sar03} and more specifically \cite[Theorem 2.3]{IomJor13}. 

\begin{proposition} \label{prop:pres}
Let $(\Sigma, \sigma)$ be the full-shift and $\psi:\Sigma \to \R$ a non-positive locally H\"older function. Then, there exists $t^* \geq 0$ such  
\begin{equation*}
P_\sigma(t \psi)=
\begin{cases}
\infty & \text{ if } t < t^*;\\
\text{finite and real analytic}   & \text{ if } t > t^*.
\end{cases}
\end{equation*}
Moreover, if $ t>t^*$ then there exists a unique equilibrium measure for $t\psi$ and it is a Gibbs measure.
\end{proposition}


When the pressure is real analytic there exists a formula for its first derivative. Indeed, if $t >t^*$ then
\begin{equation} \label{dp}
P_{\sigma}'(t \psi) \Big|_{t=t_0} = \int \psi \, d \mu_{t_0},
\end{equation}
where $\mu_{t_0}$ is the unique equilibrium measure for $t_0 \psi$, see for example  \cite[Proposition 2.6.13]{MauUrb03}. 

\subsection{The class of maps} \label{sec:map} Given a compact non degenerate interval $[a,b] \subset \R$, let $\{I_n\}_n\subset [a,b]$ be a countable collection of closed subintervals such that their interiors are pairwise disjoint and let $F:\cup_nI_n\to [a,b]$ be a map. The \emph{repeller} of the map $F$ is defined by
\begin{equation*}
\Lambda:=\{x\in [a,b] :F^j(x)\in\cup_{n}I_n \text{ for all } j\in \N\}.
\end{equation*}

For $x\in \Lambda$ there exists a sequence $(x_0, x_1, \ldots)\in \N_0^\N$ such that $F^j(x) \in I_{x_j}$.  We write $\pi(x_0, x_1, \ldots) =x$.  Note that since the $I_n$ may intersect at their boundaries, this map may not be bijective, but the set of points where this is an issue will be of zero measure, so we will abuse notation and refer to conjugacies between $\Lambda $ and shift spaces even in such a case.

We make the following assumptions on the map $F$.

\begin{enumerate}
\item[(F1)] The Hausdorff dimension of the boundary is zero, that is $\dim_H(\overline{\cup_n \partial I_n})=0$, where $\partial I$ denotes the boundary of the set $I$.
\item[(F2)] Uniform expansiveness: the map is of class $C^1$ on  $\text{int} I_n$, the interior of each interval $I_n$, for every $n \in \N$. Moreover, there exists $A>1$ such that $|F'(x)|>A$ for every $x \in \cup_n \text{int} I_n$.
\item[(F3)] The map $F$ restricted to $\Lambda$ is topologically conjugated to the full shift $(\Sigma, \sigma)$ via the projection  map $\pi: \Sigma \to \Lambda$.
\item[(F4)] The function $|F'| \circ \pi$ is locally H\"older. 
\end{enumerate}

Note that the Markov structure that we have assumed implies that there is a bijection between the space $\mathcal{M}_{\sigma}$ and that of $F$-invariant probability measures, that we denote by $\mathcal{M}_F$. For a continuous function $g: \Lambda \to \R$ we define the topological pressure of $g$ with respect to $F$ by 
    \[P_F(g):=\sup\left\{h(\nu) + \int g \, d\nu:\nu\in \mathcal M_F,\int g \, d\nu>-\infty\right\}.\]
If  $g  \circ  \pi$ is of summable variations we have that   $P_F(g)=P_\sigma(g\circ\pi)$. Therefore,  results obtained for one of the pressure are translatable to the other pressure. Hence, we identify them with the notation $P(\cdot)$.

Recall that  for each $\nu \in \mathcal{M}_F$, its \emph{Lyapunov exponent} is given by $\lambda(\nu):=\int\log|F'|\, d\nu$.
We now state our final hypothesis on the map $F$, a Bowen-type formula.  

\begin{enumerate}
\item[(F5)]
Assume  that  the function $-(\dim_H \Lambda) \log |F'|$ has zero pressure and an equilibrium state that we denote by $\mu$.
\end{enumerate}

\begin{remark}
There exist examples of maps that satisfy conditions (F1)-(F4) but not (F5), see  \cite[Example 5.3]{MauUrb96}. We have that $\dim_H(\mu)=\dim_H \Lambda$ and hence it is called measure of maximal dimension.  We stress that under these assumptions the equilibrium state for  $-(\dim_H \Lambda) \log |F'|$ is unique, this follows from \cite[Theorem 1.1]{bs}.  Note that there exist examples for which the pressure of   $-(\dim_H \Lambda) \log |F'|$ is zero and there is a corresponding Gibbs measure $\nu$ which is not an equilibrium state. This occurs when $\lambda(\nu)= \infty$, see for example \cite[p. 1757]{Sar03}.
 \end{remark}

From now on denote $b^*:=\dim_H \Lambda$. Note that since $b^*\in (0,1]$ the above implies that in Proposition \ref{prop:pres} the value of $t^*$ is less than or equal to $1$. There is a wide range of examples of maps satisfying these assumptions, we discuss some of them in Section \ref{sec:ex}.

\subsection{The observable}  \label{sec:obs}
We will be interested in studying functions that have polynomial tails with respect to the natural geometric measure $\mu$ in $\Lambda$.  
For a function $\tau: \Lambda \to \R_{\ge0}$ and  $a \in \mathcal{A}$ we denote by $\tau_a$ the restriction of $\tau$ to the cylinder $\pi[a]$.  Also, the notation  $\mu(\tau>n)$ stands for $\mu(\{x \in \Lambda : \tau(x)>n \})$.  Similarly,  given $\omega: \R_{\ge 0} \to \R_{\ge 0}$ which we discuss further below, given $n\in \N_0$,  $\mu \left(\left\{a:\inf\tau_a\in[\omega(n),\omega(n+1))\right\}\right)$  for $\mu \left(\cup [a] \right)$ where the union is over 
$\{a:\inf\tau_a\in[\omega(n),\omega(n+1))\}$. Before stating our assumptions on $\tau$, we introduce the following notion.

\begin{definition}\label{def: slowly varying func} A real valued function $\ell:\R_{\ge 0}\to \R_{\ge 0}$ is said to be \emph{slowly varying} if $\lim_{n\to\infty}\frac{\ell(\lambda n)}{\ell (n)}=1$ for all $\lambda>0$. 
 \end{definition}
\begin{remark}
A slowly varying function grows (or decays) slower than any polynomial.
If $\ell$ is slowly varying and $h$ is a monotone increasing function such that $\liminf_{n\to\infty}\frac{\log(h(n))}{n}\in(0,\infty]$, then for all $\eps>0$ and all $n$ sufficiently large, $h(n)^{-\eps}\lesssim\ell(h(n))\lesssim h(n)^\eps$.
\end{remark}

We assume that the observable $\tau$ is such that $\tau\circ\pi$ is locally H\"older  
 and there is a differentiable function $\omega:\R_{\ge 0}\to \R_{\ge 0}$, strictly increasing with $\lim_{x\to\infty}\omega(x)=\infty$ and $\frac{\omega(x+1)}{\omega(x)}\le c$ for some $c>0$ such that
\begin{enumerate}
    \item[(H1)]\label{H1} for all $\eps>0$, $1\le \#\left\{a:\inf\tau_a\in[\omega(n),\omega(n+1))\right\}\lesssim e^{\eps \omega(n)}$,
    \item[(H2)]\label{H2} there exists $\beta\in(0,1)$ and a slowly varying function  $\ell:\R_{\ge 0}\to \R_{\ge 0}$ such that for all $n \in \N$,
    $$\mu\left(\left\{a:\inf\tau_a\in[\omega(n),\omega(n+1))\right\}\right)=\ell(\omega(n))\frac{\omega'(n)}{\omega(n)^{\beta+1}},$$
    \item[(H3)]\label{H3} there exist constants $\beta_1,\beta_2\ge \beta$ satisfying the following: for all $b$ close to $b^*$, there exists $q_0>0$ such that for all $q\in[0,q_0)$, $\mu_{q,b}$ the equilibrium state for $-q\tau-b\log |F'|$, and
 all $n\in\N$,
    \begin{equation}\label{eqn: perturbations of acim} 
e^{q\omega(n)}\omega(n+1)^{-\beta_2(b^*-b)}\lesssim\frac{\mu\left(\left\{a:\inf\tau_a\in[\omega(n),\omega(n+1))\right\}\right)}{\mu_{q,b}\left(\left\{a:\inf\tau_a\in[\omega(n),\omega(n+1))\right\}\right)}\lesssim e^{q\omega(n+1)}\omega(n)^{-\beta_1(b^*-b)}.
\end{equation}
\end{enumerate}

\begin{remark}\label{rmk:Cconds} 
The following are consequences of the assumptions on the observable. 
\begin{enumerate}
 \item 
 \hyperref[H2]{(H2)} implies a non-summable polynomial tail for the observable $\tau$: for all $\eps>0$, since $\omega(n)^{-\eps}\lesssim \ell(\omega(n))\lesssim \omega(n)^{\eps},$
 \begin{align}\label{eqn: polynom tail}
\mu(\tau>n)& \ge \sum_{k=\lfloor\omega^{-1}(n)\rfloor}^\infty  \mu\left(\left\{a:\inf\tau_a\in[\omega(k),\omega(k+1))\right\}\right)  \asymp \int^\infty_{\omega^{-1}(n)}\ell(\omega(x)) \frac{\omega'(x)}{\omega(x)^{\beta+1}}dx \nonumber \\
&\asymp\int^\infty_{\omega^{-1}(n)} \frac{\omega'(x)}{\omega(x)^{\beta+1\pm\eps}}dx=-\left.\frac1{\beta\pm\eps}\omega(x)^{-\beta\pm\eps}\right|_{\omega^{-1}(n)}^\infty\asymp n^{-\beta\pm\eps}.
\end{align}
As this holds for all $\eps>0$, we conclude that $\mu(\tau>n)\asymp n^{-\beta}$, hence in particular $\int\tau~d\mu=\infty$.
\item
When $\omega(n)\asymp n^\kappa$ for some $\kappa>0$, \hyperref[H2]{(H2)} implies $\mu\left(\left\{a:\inf\tau_a\in[\omega(n),\omega(n+1))\right\}\right)=\ell(\omega(n))n^{-(\beta\kappa+1)},$ and when $\omega(n)\asymp e^{rn}$ for some $r>0$, $\mu\left(\left\{a:\inf\tau_a\in[\omega(n),\omega(n+1))\right\}\right)=\ell(\omega(n))e^{-\beta rn}.$ 
\item If $\omega(n)\asymp n^\kappa$ for some $\kappa>0$, then as $\text{var}_1(\tau\circ\pi)<\infty$, the $\omega(n+1)$ terms in \eqref{eqn: perturbations of acim} can be replaced by $\omega(n)$. This is because $\frac{n^\kappa}{(n+1)^\kappa}$ is uniformly bounded for all $n$, so $e^{q\omega(n+1)}\asymp e^{q\omega(n)}$ and $\omega(n+1)^r\asymp\omega(n)^r$ for all $r<0$.
\end{enumerate}
\end{remark}

\subsection{Simplifications of the H conditions in terms of numbers of branches} \label{sec:simple}
Under somewhat stronger assumptions we show how the conditions  \hyperref[H1]{(H1)}  \hyperref[H2]{(H2)} and  \hyperref[H3]{(H3)} are related and simplified. For a large class of systems and observables this eases the verification of the assumptions. We give two lemmas which concern the relationship between our conditions and $\#\left\{a:\tau_a\in[\omega(n),\omega(n+1))\right\}$.

\begin{definition}\label{def: comparable scaling}
   We have \emph{$\omega$-comparable $K$-scaling} if there exists $K>0$ uniform such that whenever $a, a'\in \mathcal A$ have $\inf\tau_a, \inf\tau_{a'}\in [\omega(n), \omega(n+1))$ for $n\ge 0$, then $\sup_{x\in [a], y\in [a']}\frac{|F'(x)|}{|F'(y)|}= K^{\pm1}$.
\end{definition}

\begin{lemma}\label{lemma: H3 by uniform branch count}
Suppose that there exists $K\ge 1$, such that $\#\left\{a:\tau_a\in[\omega(n),\omega(n+1))\right\}\in[1,K]$, and we have $\omega$-comparable $K$-scaling with $\omega(n)= n^\kappa$ for $\kappa>0$.  Then \hyperref[H2]{(H2)} implies \hyperref[H3]{(H3)} (with $\beta_{1,2}=\frac{\beta+1/\kappa}{b^*}$).
\end{lemma}

\begin{proof}
In this case, for any $x\in [a]$ where $\inf\tau_a\in[\omega(n),\omega(n+1))$, for $C_\mu$ the Gibbs constant for $\mu$,
$$C_\mu^{-1}|F'(x)|^{-b^*}\le \mu\left(\left\{a:\inf\tau_a\in[\omega(n),\omega(n+1))\right\}\right)\le C_\mu K^2|F'(x)|^{-b^*},$$
by \hyperref[H2]{(H2)} and Remark~\ref{rmk:Cconds}(2) we have for all $\eps>0$ small, $|F'(x)|^{-b^*}\asymp \frac{1}{n^{\kappa\beta+1\pm\eps}}$ and 
$$\mu_{q,b}\left(\left\{[a]:\inf\tau_a\in[\omega(n),\omega(n+1))\right\}\right)\lesssim e^{-q\omega(n)} |F'(x)|^{-b}\asymp e^{-q\omega(n)}n^{-(\kappa\beta+1\pm\eps)b/b^*}.$$
Therefore taking $\eps$ to 0 and the ratio of measures,
\begin{equation*}
    \frac{\mu\left(\left\{a:\inf\tau_a\in[\omega(n),\omega(n+1))\right\}\right)}{\mu_{q,b}\left(\left\{a:\inf\tau_a\in[\omega(n),\omega(n+1))\right\}\right)}\gtrsim e^{q\omega(n)}n^{-(\kappa\beta+1)(1-b/b^*)}=e^{q\omega(n)}\omega(n)^{-\beta'(b^*-b)},
\end{equation*}
where $\beta'=\frac{\beta+1/\kappa}{b^*}$. Similarly, since 
$$\mu_{q,b}\left(\left\{[a]:\inf\tau_a\in[\omega(n),\omega(n+1))\right\}\right)\gtrsim  e^{-q\omega(n)}n^{-(\kappa\beta+1\pm\eps)b/b^*},$$
one gets
$$\frac{\mu\left(\left\{a:\inf\tau_a\in[\omega(n),\omega(n+1))\right\}\right)}{\mu_{q,b}\left(\left\{a:\inf\tau_a\in[\omega(n),\omega(n+1))\right\}\right)}\lesssim e^{q\omega(n)}\omega(n)^{-\beta'(b^*-b)},$$
hence \hyperref[H3]{(H3)} holds with $\beta_{1,2}=\beta'$.
\end{proof}
 
 \begin{lemma}\label{lemma: branch count by H3}
 Suppose that both  \hyperref[H2]{(H2)} and \hyperref[H3]{(H3)} hold and we have $\omega$-comparable $K$-scaling.  Then if $\omega(n)$ is polynomial (of order $\kappa$), let $c_{1,2}=\beta_{1,2}b^*-\beta$
 $$\omega'(n)\omega(n)^{c_1-1}\lesssim\#\left\{a:\tau_a\in[\omega(n),\omega(n+1))\right\}\lesssim\omega'(n)\omega(n)^{c_2-1}.$$ 
 \end{lemma}

 \begin{proof}
 Let $c_n:= \#\left\{a:\tau_a\in[\omega(n),\omega(n+1))\right\}$ and let $a_n:= |F'(x)|^{-b^*}$ for $x\in [a]$ for some $a\in \mathcal A$ with $\inf\tau_a\in[\omega(n),\omega(n+1))$.  Then our \hyperref[H2]{(H2)} implies that 
 \begin{equation}
 \label{eqn: H2 with comparable scaling}
 \frac{\omega'(n)}{\omega(n)^{\beta+1\pm\eps}}\asymp\mu\left(\left\{[a]:\inf\tau_a\in[\omega(n),\omega(n+1)\right\}\right) \asymp c_na_n
 \end{equation}
 Moreover, if $\omega(n)$ is polynomial,
 \begin{equation}
 e^{-q\omega(n+1)}c_na_n^{b/b^*}\lesssim\mu_{q,b}\left(\left\{[a]:\inf\tau_a\in[\omega(n),\omega(n+1))\right\}\right) \lesssim e^{-q\omega(n)} c_na_n^{b/b^*}.
 \label{eq:branchcount}
 \end{equation}
First by \hyperref[H3]{(H3)}, 
$$\frac{c_na_n}{e^{-q\omega(n)}c_na_n^{b/b^*}}\lesssim e^{q\omega(n+1)}\omega(n)^{-\beta_1(b^*-b)}, \ \text{\ie} a_n\lesssim{\omega(n)^{-\beta_1 b^*}},$$
putting this into \eqref{eqn: H2 with comparable scaling} and taking $\eps\to0$, we get 
\begin{equation*}c_n\gtrsim \omega'(n)\omega(n)^{\beta_1 b^*-\beta-1}.\end{equation*}
Similarly by \hyperref[H3]{(H3)} there is $a_n\gtrsim \omega(n)^{-\beta_2b^*}$, and put this into \eqref{eqn: H2 with comparable scaling} and taking $\eps\to0$ we get the upper bound
$c_n\lesssim \omega'(n)\omega(n)^{\beta_2b^*-\beta-1}$.
 \end{proof}

Note that above, the fact that $\omega(n)$ is polynomial was only used to ensure that \eqref{eq:branchcount} holds. If it were exponential, then the discrepancy between $e^{q\omega(n)}$ and $e^{q\omega}$ for $\omega\in (\omega(n), \omega(n+1))$ may cause issues.  But we shall see, for example in Example~\ref{ex:exponential} that this is not always the case.

\section{Real analyticity of the spectrum}
In this section we prove that despite the fact that the multifractal decomposition induced by the Birkhoff averages is extremely complicated (each level set is dense in the repeller), the function that encodes it, $b(\cdot)$, is as regular as possible (real analytic). In order to do so we make use of the associated  thermodynamic formalism. Similar results have been obtained  for other dynamical systems or different classes of observables (see, for example, \cite{IomJor15, MauUrb03, pes}).

Recall that $F$ restricted to $\Lambda$ is topologically conjugated to the  full-shift on $\mathcal A=\N$; abusing notation, $\sup_{x\in[a]}$ is to be understood as taking supremum over $x$ in $\pi[a]$. The following function is our key to prove Theorem~\ref{thm: main theorem} and the real analyticity of $b(\cdot)$:
\begin{equation}\label{eqn: def of p function}
    p:\R^3\to \R, \quad \hspace{3pt}p(\alpha,q,b):=P(q(\alpha-\tau)-b\log|F'|).
\end{equation}

Let $\mathcal S:=\left\{(q,b)\in\R^2:P(-q\tau-b\log|F'|)<\infty\right\}$ denote the set where the function $p$ is finite.

\begin{lemma}\label{lemma: parameter range of finite pressure}
We have that
$$\mathcal S=\left(\{0\}\times[b^*,\infty)\right)\cup\left((0,\infty)\times[0,\infty)\right),$$
where $b^*=\dim_H(\Lambda)$ is the unique solution to $P(-t\log|F'|)=0$ (which exists by assumption (F5)). 
\end{lemma}
\begin{proof}
Suppose $q>0$.   By \hyperref[H1]{(H1)}, for $\eps\in (0, q)$,
\begin{equation*}\label{ineqn: sum of exp(sup)}
\sum_{a\in\mathcal A}\exp\left(\sup-q\tau_a(x)\right)\le\sum_{k\in\N}\sum_{\{a:\inf\tau_a\in[\omega(k),\omega(k+1))\}}\exp(-q\omega(k))
\lesssim\sum_{k\in \N} \exp\left(-(q-\eps)\omega(k)\right)< \infty,
\end{equation*}
 and by Lemma~\ref{lemma: MauUrb}, $P(-q\tau)<\infty$. Since $F$ is uniformly expanding, for all $b\ge0$  we have that $P(-q\tau-b\log|F'|)\le P(-q\tau)<\infty$.

Now assume that $q<0$, then for any $b\in\R$, the Variational Principle implies
\begin{equation*}
    P(-q\tau-b\log|F'|)\ge h(\mu)+\int\left(-q\tau-b\log|F'|\right)d\mu\ge-q\int \tau\,d\mu-b\lambda(\mu).
\end{equation*}
By assumption $\int\tau\,d\mu=\infty$ (see Remark~\ref{rmk:Cconds}(1)), therefore $(-\infty,0)\times\R\not\in\mathcal S$.

Now for $q=0$, we assumed in Section \ref{sec:map}  that $P( -(\dim_H \Lambda) \log|F'|)=0$.
By \cite[Theorem 3.1]{Iom05}, 
\begin{equation}\label{eqn: def of b*}P(-b\log|F'|)= \left\{\begin{aligned}
    &\text{non positive} \hspace{3mm}& b>b^*;\\
    &\infty &b<b^*.
\end{aligned}
\right.\end{equation}
The result then follows.
\end{proof}

\begin{remark}\label{rmk: Gibbs measures for S}
Since $\sum_{j\ge1}\text{var}_j(\tau\circ\pi)$, $\sum_{j\ge1}\text{var}_j(\log|F'|\circ\pi)<\infty$ and $(\Sigma,\sigma)$ is a full-shift by \cite[Theorem 1]{Sar03} for each $(q,b)\in\mathcal S$ there exists a unique equilibrium state for the potential $-q\tau-b\log|F'|$. Moreover,  it is a Gibbs measure. 
\end{remark}

Let $\underline\alpha=\inf\left\{\lim_{n\to\infty}\frac1n\sum_{j=0}^{n-1}\tau(F^jx):x\in\Lambda\right\}\in(0,\infty]$.
The dimension of the level sets is obtained by the following variational formula.

\begin{proposition} \label{prop: dim=sup of invariant measures} \cite[Theorem 3.1]{IomJor15} For each $\alpha>\underline\alpha$,
$$\dim_H(J(\alpha))=\sup\left\{\frac{h(\nu)}{\lambda(\nu)}:\nu\in\mathcal M_F,\lambda(\nu)<\infty,\int\tau\,d\nu=\alpha\right\}.$$
\end{proposition}

The following three lemmas and proposition set up the proof of Theorem~\ref{thm: main theorem} in the next section. 

\begin{lemma}\label{lemma: pressure is non-negative for correct b}
  If $\alpha>\underline\alpha$ and  $q>0$ then $p(\alpha,q,b(\alpha))\ge0$.
\end{lemma}

\begin{proof}
    The proof is similar to that of \cite[Lemma 4.3]{IomJor15}. It is a direct consequence of  Proposition~\ref{prop: dim=sup of invariant measures} that there exists a sequence of $F-$invariant probability measures $(\mu_n)_n$  such that  for all $n\in\N$, $\int\tau d\mu_n=\alpha$, $h(\mu_n),\lambda(\mu_n)<\infty$, and $\lim_{n\to\infty}h(\mu_n)/\lambda(\mu_n)=b(\alpha)$. Let $0<s_1<s_2<b(\alpha)$ and $q>0$, then by Lemma~\ref{lemma: parameter range of finite pressure}, for $K:=P(-q\tau-s_1\log|F'|)<\infty$, and by the Variational Principle, 
    $$h(\mu_n)-q\alpha-s_1\lambda(\mu_n)=h(\mu_n)+\int \left(-q\tau-s_1\log|F'| \right)\, d\mu_n\le K.$$
    For all $n$ large enough we have $s_2\le h(\mu_n)/\lambda(\mu_n)\le b(\alpha)$. Hence,  combining these we get for all $n$ large $\lambda(\mu_n)\le (K+q\alpha)/(s_2-s_1)<\infty$, so the sequence $\{\lambda(\mu_n)\}_n$ is bounded.
    
 By the Variational Principle, since $q>0$, we have that  for all $n\in\N$, 
 $$p(\alpha,q,b(\alpha))\ge h(\mu_n)+\int(q(\alpha-\tau)-b(\alpha)\log|F'|) \, d\mu_n.$$ 
 Hence, by our choice of $\{\mu_n\}_n$ and boundedness of $\{\lambda(\mu_n)\}_n$,
    \begin{equation*}p(\alpha,q,b(\alpha))
    \ge\lim_{n\to\infty} \left( h(\mu_n)+q\left(\alpha-\int\tau \, d\mu_n\right)-b(\alpha)\lambda(\mu_n) \right) \ge\lim_{n\to\infty}\left( \lambda(\mu_n)\left(\frac{h(\mu_n)}{\lambda(\mu_n)}-b(\alpha)\right) \right)=0.\qedhere
    \end{equation*}
   \end{proof}

\begin{lemma}\label{lemma: inf and divergence of pressure}
   For each $\alpha>\underline \alpha$ we have that  $\inf\left\{p(\alpha,q,b(\alpha)):q>0\right\}=0$ and $\lim_{q\to\infty}p(\alpha,q,b(\alpha))=\infty$.
\end{lemma} 

\begin{proof}  
    Suppose $\inf\left\{p(\alpha,q,b(\alpha)):q>0\right\}=C>0$, then, using the idea of Lemma~\ref{lemma: approximation property of pressure},  \cite[Lemma 3.2]{IomJorTod17} implies that there exists a compact invariant set $M \subset \Lambda$ such that (i) $p_M(\alpha,q,b(\alpha))>0$ for all $q\in\R$ and (ii) $\lim_{|q|\to\infty}p_M(\alpha,q,b(\alpha))=\infty$, where $p_M(\alpha,q,b):=P_M(q(\alpha-\tau)-b\log|F'|)$.

    By analyticity and Ruelle's formula (see, for example, \cite[Proposition 2.6.13]{MauUrb03}), (ii) implies there exists $q_M$ such that 
    \[0=\left.\frac\partial{\partial q}\right|_{q=q_M}p_M(\alpha,q,b(\alpha))=\int(\alpha-\tau) \, d\mu_{q_M}\] where $\mu_{q_M}$ is the equilibrium state corresponding to $q_M(\alpha-\tau)-b(\alpha)\log|F'|$ for the restriction of $F$ to $M$. Therefore $\int \tau \, d\mu_{q_M}=\alpha$, but 
    \[0<p_M(\alpha,q_M,b(\alpha))=h(\mu_{q_M})-b(\alpha)\int \log|F'| \,d\mu_{q_M}=\lambda(\mu_{q_M})\left(\frac{h(\mu_{q_M})}{\lambda(\mu_{q_M})}-b(\alpha)\right),\]
    which is a contradiction to Proposition~\ref{prop: dim=sup of invariant measures}, so combining this with Lemma~\ref{lemma: pressure is non-negative for correct b}, $\inf\{p(\alpha,q,b(\alpha)):q>0\}=0$  where $p_M$ is the topological pressure restricted to $M$.
    
In order to prove the second statement, as in the proof of \cite[Lemma 3.2(2)]{IomJorTod17} with $\psi=1$, there exists a compact $F$-invariant set $M\subset \Lambda$ and a measure $\tilde\nu$ supported on the orbit of a periodic point $\tilde x\in M$ such that $\alpha>\int \tau \, d\tilde\nu $, then by the Variational Principle
\[\lim_{q\to\infty}p(\alpha,q,b(\alpha))\ge\lim_{q\to\infty}p_M(\alpha,q,b(\alpha))\ge \lim_{q\to\infty} \left(\int q(\alpha-\tau) \, d\tilde\nu-b(\alpha)\lambda(\tilde\nu) \right)=\infty.\qedhere\]
\end{proof}

\begin{lemma}\label{lemma: root and partial derivative}
    For all $\alpha>\underline{\alpha}$, there exists $q(\alpha)\in(0,\infty)$ such that 
    \begin{equation}
\left.\frac\partial{\partial q}\right|_{q=q(\alpha)}p(\alpha, q,b(\alpha))=0.
    \end{equation}
    In particular, $p(\alpha,q(\alpha),b(\alpha))=0.$
\end{lemma}

\begin{proof}
We first show that the $q(\alpha)$ exists and is strictly positive. By Lemmas~\ref{lemma: parameter range of finite pressure} and \ref{lemma: inf and divergence of pressure}, for $q_0$ large,
\[\left.\frac\partial{\partial q}\right|_{q=q_0}p(\alpha,q,b(\alpha))>0.\]
Suppose there is no $q(\alpha)>0$ such that the partial derivative with respect to $q$ is 0. Then $p(\alpha,q,b(\alpha))$ is strictly increasing on $(0,\infty)$: if there exists an open set on which the pressure is decreasing then the derivative is negative, by analyticity and strict convexity of the pressure when it is finite, there exists a zero for derivative. Therefore by Lemma~\ref{lemma: inf and divergence of pressure}, $\lim_{q\to0^+}p(\alpha,q,b(\alpha))=\inf\left\{p(\alpha,q,b(\alpha)):q>0\right\}=0$. 

Again by Lemma~\ref{lemma: parameter range of finite pressure} $p(\alpha,1/n,b(\alpha))<\infty$ which implies for each $n$ there is a unique equilibrium (Gibbs)  state $\nu_n$ with respect to the potential $\frac1n(\alpha-\tau)-b(\alpha)\log|F'|$, and as $var_j(\tau/n)\le var_j(\tau)$ for each $j\ge1$, these Gibbs measures share the same Gibbs constant.

Also by assumption \hyperref[H2]{(H2)} and Remark~\ref{rmk:Cconds}, finite 1 variations of $F'\circ\pi$ and the Gibbs property of $\mu$, for $k\in\N$ large enough
\begin{equation}\label{eqn: measure and distortion}
    k^{-\beta}\asymp\mu(\tau>k)\lesssim \sum_{n\in\N:
    \omega(n)>k}\sum_{\left\{a:\,\inf\tau_a\in[\omega(n),\omega(n+1))\right\}}\sup_{x\in[a]}|F'(x)|^{-b^*}.
\end{equation}
Hence for all $n\in\N$, by Ruelle's formula, for all $k$ large enough, 
\begin{align*}
      &\left.\frac\partial{\partial q}\right|_{q=1/n}p(\alpha,q,b(\alpha))=\int(\alpha-\tau)\, d\nu_n=\alpha+\sum_{a\in\mathcal A}\int_{[a]}-\tau\,d\nu_n\le\alpha+\sum_{a\in\mathcal A}-\inf\tau_a\nu_n([a])\\
      &\lesssim \sum_{a\in\mathcal A}-\inf\tau_a\sup_{x\in[a]}\exp\left(\frac{\alpha-\tau(x)}{n}-b(\alpha)\log|F'(x)|\right)\le \sum_{a\in\mathcal A}-\inf\tau_a\sup_{x\in[a]}{\exp\left(\frac{\alpha-\tau(x)}{n}\right)}{|F'(x)|^{-b^*}}\\\
      &\lesssim-k\exp\left(\frac{(\alpha-k)}{n}-p\left(\alpha,\frac1n,b(\alpha)\right)\right)\sum_{j\in\N:\omega(j)\ge k}\sum_{\left\{a\in\mathcal{A}:\inf\tau_a\in[ \omega(j),\omega(j+1))\right\}}\mu([a])\\
      &\lesssim -k^{1-\beta}\exp\left(\frac{(\alpha-k)}{n}-p\left(\alpha,\frac1n,b(\alpha)\right)\right),
\end{align*}
where the second to the last asymptotic inequality is by \hyperref[H2]{(H2)} and Remark~\ref{rmk:Cconds}. By continuity of the pressure function and $p(\alpha,1/n,b(\alpha))\to0$, for each all $n,k\in\N$ large enough
\[\left.\frac\partial{\partial q}\right|_{q=1/n}p(\alpha,q,b(\alpha))\lesssim -k^{1-\beta}<0,\]
which can be arbitrary large, implying there exists a neighbourhood to the right of 0 such that the derivative is negative, again this is a contradiction. Hence there exists $q(\alpha)>0$ where the derivative of the pressure functional with respect to $q$ is zero.
 
Now we prove the pressure at $q(\alpha)>0$ is also zero. Let $\mu_{\alpha}$ denote the equilibrium Gibbs state for the potential $q(\alpha)(\alpha-\tau)-b(\alpha)\log|F'|$. We have shown already that
    \[\left.\frac\partial{\partial q}\right|_{q=q(\alpha)}p(\alpha,q,b(\alpha))=\int(\alpha-\tau) \,  d\mu_{\alpha}=0,\]
in other words, $\int\tau\,d\mu_{\alpha}=\alpha$, so by Proposition~\ref{prop: dim=sup of invariant measures}, 
    \[p(\alpha,q(\alpha),b(\alpha))=\left(\frac{h(\mu_{\alpha})}{\lambda(\mu_{\alpha})}-b(\alpha)\right)\lambda(\mu_\alpha)\le0.\]
Combining this inequality with Lemma~\ref{lemma: pressure is non-negative for correct b}, $p(\alpha,q(\alpha),b(\alpha))=0$.
\end{proof}

\begin{proposition}\label{lemma: q b are analytic}
    The functions $b(\alpha)$ and $q(\alpha)$ are real-analytic with respect to $\alpha\in(0,\infty)$, and $b(\alpha)$ is strictly increasing with $\lim_{\alpha\to\infty}b(\alpha)=dim_H(\Lambda)=b^*$.
\end{proposition}

\begin{proof}
    For analyticity, the method provided in \cite[Lemma 4.5]{IomJor15} allows us to conclude that the Jacobian matrix (with respect to $q,b$) of the mapping
\[G:\R\times \R^2\to \R^2, \hspace{3mm}G(\alpha,q,b)=\left(\begin{array}{c}
    p(\alpha,q,b)\\\frac{\partial}{\partial q}p(\alpha,q,b)
\end{array}\right)\]
  evaluated at $q=q(\alpha)$, $b=b(\alpha)$ has strictly positive determinant, for all $\alpha>\underline{\alpha}$. Hence, analyticity of $q(\alpha)$ and $b(\alpha)$ follows from the Analytic Implicit Function Theorem because the pressure function is analytic when it is finite (see \cite[Theorem 3]{Sar01}). Moreover,
\begin{equation*}
    \left(\begin{array}{c}q'(\alpha)\\b'(\alpha)\end{array}\right)=-\left(\begin{array}{c}\frac{\partial p}{\partial q}\\\frac{\partial^2p}{\partial q^2}\end{array}
    \begin{array}{c}\frac{\partial p}{\partial b}\\\frac{\partial^2 p}{\partial b\partial q}\end{array}\right)^{-1}\left(\begin{array}{c}\frac{\partial p}{\partial \alpha}\\\frac{\partial^2 p}{\partial \alpha\partial q}\end{array}\right)
\end{equation*}
evaluated at $\alpha, q(\alpha),b(\alpha)$. Since  $\frac{\partial p}{\partial q}=0$ at $q(\alpha)$, $\frac{\partial^2 p}{\partial q^2}$ is strictly positive by convexity of pressure function, and $\left.\frac{\partial p}{\partial b}\right|_{\alpha,q(\alpha),b(\alpha)}=-\int|\log|F'|d\mu_\alpha$ where $\mu_\alpha$ is the Gibbs measure for $q(\alpha)(\alpha-\tau)-b(\alpha)\log|F'|$, we have
\begin{equation}\label{eqn: b' alpha}
    b'(\alpha)=\left({\frac{\partial^2 p}{\partial q^2}\cdot\frac{\partial p}{\partial b}}\right)^{-1}\left(-\frac{\partial^2 p}{\partial q^2}q(\alpha)+\frac{\partial p}{\partial q}\frac{\partial^2p}{\partial\alpha\partial q}\right)=q(\alpha)/\lambda(\mu_\alpha).
\end{equation}

 Since $F$ is uniformly expanding, $\lambda(\mu_\alpha)>0$, and $q(\alpha)>0$ by Lemma~\ref{lemma: root and partial derivative}, $b(\alpha)$ is strictly increasing.
\end{proof}

\section{Proof of Theorem~\ref{thm: main theorem}}

\begin{proposition}\label{prop: alpha q alpha goes to 0} The following holds, 
    \[\lim_{\alpha\to\infty}q(\alpha)=0 \quad \text{ and } \quad \lim_{\alpha\to\infty}\alpha q(\alpha)=0.\]
\end{proposition}

\begin{proof}
  To show the first equality, let $q_{\infty}:=\limsup_{\alpha\to\infty}q(\alpha)$ and assume by contradiction that $q_{\infty}\in(0,\infty)$. Then, there exists $\{\alpha_n\}_n$ strictly increasing such that $\limsup_{n\to\infty} \alpha_n=\infty$ and $\lim_{n\to\infty} q(\alpha_n)=q_\infty>0$. By Lemma~\ref{lemma: root and partial derivative}, Lemma~\ref{lemma: q b are analytic} and the continuity of the pressure function on $\mathcal S$, for all $\alpha>\underline\alpha$, we have that
    $$-q(\alpha)\alpha=P(-q(\alpha)\tau-b(\alpha)\log|F'|)\in\R.$$
     By continuity,
    $$\lim_{n \to \infty} P(-q(\alpha_n)\tau-b(\alpha_n)\log|F'|)=P(-q_\infty\tau-b^*\log|F'|)$$
     which is finite. However,  $\lim_{n\to\infty}-\alpha_nq(\alpha_n)=-\infty$, which is a contradiction.

    Similarly, suppose $q_\infty=\infty$: as in the proof of Lemma~\ref{lemma: inf and divergence of pressure}, there exists a periodic point $\tilde x$ of period $p$ such that for all $n$ large enough, $\frac1p\sum_{j=0}^{p-1}\tau(F^j \tilde x)<\alpha_n$ \ie $\lim_{n\to\infty}\alpha_n-\frac1p\sum_{j=0}^{p-1}\tau(F^j\tilde x)=\infty$. Denote by $\tilde\nu$ the probability measure supported on the orbit of $\tilde x$.     
     Since $b(\alpha_n)$ is bounded from above for all $n \in \N$, hence by the definition of $q(\alpha_n)$ and Lemma~\ref{lemma: root and partial derivative}, we have that
         \[0=\lim_{n\to\infty}p(\alpha_n,q(\alpha_n),b(\alpha_n))\ge   \lim_{n\to\infty}\left(  \int q(\alpha_n)(\alpha_n-\tau) \, d\tilde\nu-b(\alpha_n)\log|(F^p)'(\tilde x)| \right)=\infty,\]
    which is a contradiction. Therefore,  $\lim_{\alpha \to \infty}q(\alpha)=0$.

    Lastly, by continuity, as we have shown that $q(\alpha)\to0$, by definition of $b^*$ \eqref{eqn: def of b*},
    \[\lim_{\alpha\to\infty}q(\alpha)\alpha=-\lim_{\alpha\to\infty}p(-q(\alpha)\tau-b(\alpha)\log|F'|)=P(-b^*\log|F'|)\le0.\]
     But since both $q(\alpha)$ and $\alpha$ are strictly positive, $\lim_{\alpha\to\infty}\alpha q(\alpha)=0$.
    \end{proof}

\begin{lemma}\label{lemma: gap is comparable to integral of q(t)}
    There exists $L>0$ such that for all $\alpha>L$,
    \begin{equation}\label{eqn: b*-b asymptotics}
     {b^*-b(\alpha)}\asymp{\int_\alpha^\infty q(t)dt}.
    \end{equation}
\end{lemma}

\begin{proof}

For all $\alpha>\underline\alpha$, as $b^*=\lim_{\alpha\to\infty}b(\alpha)$, it is easy to deduce from \eqref{eqn: b' alpha} that $b^*-b(\alpha)=\int_\alpha^\infty q(t)/\lambda(\mu_t)dt$. 
 Since $F$ is uniformly expanding the Lyapunov exponent $\lambda(\mu_\alpha)$ is uniformly bounded from below, hence it suffices now to find a uniform upper bound for $\lambda(\mu_\alpha)$.

Denote by $C_{\alpha}$ the Gibbs constant for the measure $\mu_{\alpha}$; it is well-known that $C_\alpha$ only depends on the variations of the potential $-q(\alpha)\tau-b(\alpha)\log|F'|$, as $\tau$ is of summable variations and $\text{var}_j\left(b(\alpha)\log|F'|\right)\le \text{var}_j(\log|F'|)$ for all $j\in\N$ so as $\alpha\to\infty$ the constants $C_\alpha$ are uniformly bounded. 

By Lemma~\ref{lemma: root and partial derivative} $p(\alpha,q(\alpha),b(\alpha))=0$, so the Gibbs property gives $\mu_\alpha([a])\asymp e^{q(\alpha)(\alpha-\tau_a)-b(\alpha)\log|F_a'|}$ for all $a\in\mathcal A$. Also we have $b(\alpha)\to b^*$ as $\alpha\to\infty$, so there exists $L>0$ such that for all $\alpha>L$, $|b^*-b(\alpha)|<\frac{\beta}{8\beta_2}$. Since for all $y>1$ there is $\log y\le y^\eta/\eta$, applying \hyperref[H3]{(H3)}, for all $0<\eta<\min\left\{\frac{b^*}2,\frac{\beta}{8 \beta_2},\frac{\beta}8\right\}$, as $\omega(n+1)\le c\omega(n)$ and $\ell(\omega(n))\lesssim\omega(n)^{\eta}$,
\begin{align*}
    \int\log|F'|d\mu_{\alpha}&\asymp e^{\alpha q(\alpha)}\sum_{a\in\mathcal A}\log|F'_a|\mu_{q(\alpha),b(\alpha)}([a])\lesssim\sum_{a\in\mathcal A}e^{-q(\alpha)\tau_a-(b(\alpha)-\eta)\log|F'_a|}=\sum_{a\in\mathcal A}\mu_{q(\alpha),b(\alpha)-\eta}([a])\\
    &\lesssim\sum_{n\in\N}e^{-q(\alpha)\omega(n)}\omega(n+1)^{\beta_2(b^*-b(\alpha)+\eta)}\frac{\ell(\omega(n))\omega'(n)}{\omega(n)^{\beta+1}}\lesssim \sum_{n\in\N}(c\,\omega(n))^{\beta/4}\frac{\omega'(n)}{\omega(n)^{1+\beta-2\eta}}\\
    &\lesssim \sum_{n\in\N}\frac{\omega(n)'}{\omega(n)^{1+\beta-\beta/2}}
    \lesssim\int_1^\infty\frac{\omega(x)}{\omega(x)^{1+\beta/2}}=-\left. \frac2\beta \omega(x)^{-\beta/2}\right|^\infty_0=\frac2\beta\omega(0)^{-\frac\beta2}<\infty.
\end{align*}
Hence $\int\log|F'|d\mu_\alpha$ is uniformly bounded from above.\qedhere

\end{proof}

\begin{proposition}\label{prop: asymptotics of alpha over q-alpha}
For every $\eps>0$ small enough we have that
   $$q(\alpha)^{-(1-\beta-\eps)}\lesssim \alpha\lesssim q(\alpha)^{-(1-\beta+\eps)}.$$
\end{proposition}

\begin{proof}
As in the proof of Lemma~\ref{lemma: pressure is non-negative for correct b}, Lemma~\ref{lemma: root and partial derivative} and Lemma~\ref{lemma: gap is comparable to integral of q(t)} above, for each $\alpha>\underline\alpha$ there exists a unique Gibbs equilibrium state $\mu_\alpha$ for $q(\alpha)(\alpha -\tau) -b(\alpha) \log|F'|$ and $\int\tau\,d\mu_\alpha=\alpha$. 
 By assumption there is $c>0$ such that $\frac{\omega(x+1)}{\omega(x)}\le c$.
  For arbitrary $\eps>0$ there is $L_\eps\ge L$ such that for all $\alpha>L_\eps$, $b^*-b(\alpha)<\min\left\{\frac{\eps}{2\beta_1},\frac{\eps}{2\beta_2}\right\}$,  by the Gibbs property, the regularity of $\tau$ (especially finite 1-variation) assumption \hyperref[H2]{(H2)} and \hyperref[H3]{(H3)}, since $\ell(\omega(n))\lesssim \omega(n)^{\eps/2}$, 
\begin{align*}
    \alpha&=\sum_{a\in\mathcal A}\int_{[a]}\tau d\mu_{\alpha}\lesssim\sum_{n\in\N}\omega(n+1)\sum_{\{a:\inf\tau_a\in[\omega(n),\omega(n+1))\}}\mu_\alpha([a])\\
    &\lesssim \sum_{n\in\N}c\omega(n)e^{\alpha q(\alpha)}(c\,\omega(n))^{\beta_2(b^*-b(\alpha))} e^{-q(\alpha) \omega(n)}\frac{\ell(\omega(n))\omega'(n)}{\omega(n)^{\beta+1}}\\
    &\lesssim e^{\alpha q(\alpha)}\sum_{n\in\N}e^{-q(\alpha)\omega(n)}\omega'(n)\omega(n)^{-(\beta-\eps)}\asymp e^{\alpha q(\alpha)}\int_0^\infty e^{-q(\alpha)\omega(x)}\omega'(x)\omega(x)^{-(\beta-\eps)}dx
\end{align*}

Then substituting $s=q(\alpha)\omega(x)$ and recalling, by Proposition~\ref{prop: alpha q alpha goes to 0}, that $\alpha q(\alpha)\to0$ as $\alpha\to\infty$, we obtain that for all $\eps>0$ and $\alpha$ large enough,
\begin{equation}
    \alpha\lesssim e^{\alpha q(\alpha)}\int_0^\infty\frac1{q(\alpha)}e^{-s}\left(\frac s{q(\alpha)}\right)^{-(\beta-\eps)}\asymp q(\alpha)^{-(1-\beta+\eps)}\Gamma(1-\beta+\eps)
\end{equation}
where $\Gamma(\cdot)$ denotes the usual Gamma function, $\Gamma(r)=\int_0^\infty x^{r-1}e^{-x}dx$. Since $\beta\in(0,1)$, $\Gamma(1-\beta+\eps)$ is bounded away from infinity hence $\alpha\lesssim q(\alpha)^{-(1-\beta+\eps)}$.

Now we prove the asymptotic lower bound. Using \hyperref[H3]{(H3)} and the substitution $s=c\,q(\alpha)\omega(x)$,
\begin{align*}
    \alpha&\gtrsim\sum_{n\in\N}\omega(n)\mu_\alpha([a])\gtrsim \sum_{n\in\N}e^{-q(\alpha)\omega(n+1)}\ell(\omega(n))^{-1}\omega(n)^{1+\beta_1(b^*-b(\alpha))}\omega'(n)\omega(n)^{-(1+\beta)}\\
    &\gtrsim\sum_{n\in\N}e^{-q(\alpha)c\omega(n)}\omega'(n)\omega(n)^{-(\beta+\eps)}\asymp \int_0^\infty\frac1{c\,q(\alpha)}e^{-s}\left(\frac s{c\,q(\alpha)}\right)^{-(\beta+\eps)}ds\\
    &\asymp q(\alpha)^{-(1-\beta-\eps)}\Gamma(1-\beta-\eps).
\end{align*}

For all $\eps<\frac{1-\beta}{2}$, $\Gamma(1-\beta-\eps)$ is uniformly bounded so we get $\alpha \gtrsim q(\alpha)^{-(1-\beta-\eps)}$.
\end{proof}

\begin{proof}[Proof of Theorem~\ref{thm: main theorem}]
Let $\eps>0$, by Lemma~\ref{lemma: gap is comparable to integral of q(t)}, Proposition~\ref{prop: asymptotics of alpha over q-alpha} for all $\alpha\in(L,\infty)$ we have the following asymptotic inequality:
\begin{equation}\label{ineqn: upper bound of convergence}
b^*-b(\alpha)\asymp\int_{\alpha}^\infty q(t)\ dt\lesssim \int_{\alpha}^\infty t^{-\frac1{1-\beta+\eps}} \ dt\asymp \alpha^{-\frac{\beta-\eps}{1-\beta+\eps}}.
\end{equation}

As the above asymptotic inequality holds for all $\eps>0$ small, we obtain that  if $x<\frac{\beta}{1-\beta}$ then, 
 \[\lim_{\alpha\to\infty}\left(b^*-b(\alpha)\right)\alpha^x=0.\]

Similarly, $$b^*-b(\alpha)\gtrsim\int_\alpha^\infty t^{-\frac1{1-\beta-\eps}}\,dt=\alpha^{-\frac{\beta+\eps}{1-\beta-\eps}},$$ 
so if $x>\frac{\beta}{1-\beta}$ then $\lim_{\alpha\to\infty}\left(b^*-b(\alpha)\right)\alpha^x=\infty$.
\end{proof}


\section{Examples} \label{sec:ex}
In this section we provide examples that illustrate our results.  Our primary example is the induced system of the  non-uniform hyperbolic interval map defined by Pomeau and Manneville.

\begin{example}[Manneville-Pomeau map]\label{ex1}
Let $\lambda>1$ and define the \emph{Manneville-Pomeau map} $f_\lambda:[0,1]\to[0,1]$ by
\begin{equation*}
    f_\lambda(x)=\left\{\begin{aligned}
       &x(1+2^\lambda x^\lambda) &x\in[0,1/2),\\
       &2x-1&x\in[1/2,1].
    \end{aligned}\right.
\end{equation*}
Let $\tau:(1/2, 1]\to\N$ be the first return time map to $[1/2, 1]$, $\tau(x):=\inf\left\{j\ge1:f^j(x)\in [1/2,1]\right\}$ and let $F:(1/2, 1] \to (1/2, 1]$ be the map  defined by $F=f^{\tau}$.  The map $F$ has an invariant probability measure that is absolutely continuous with respect to the Lebesgue measure and that we denote by $\mu$. It is well known that the map $F$, together with  $\mu$,  satisfies (F1)--(F5) and $b^*=1$. Consider the observable defined by $\tau$. Clearly $\text{var}_j(\tau\circ\pi)=0$ for all $j\ge1$, and the relevant $\omega(\cdot)$ here is $\omega(n)=n$, and \hyperref[H1]{(H1)} holds since  $\#\left\{\tau=n\right\}=1$.  By, for example the proof of \cite[Proposition 2]{BruTod08}, we have that $\mu(\tau=n)= \ell(n)n^{-(1+\frac1\lambda)}$ (note that the calculation referenced gives $\ell(n)$ of the form $1+c\log n/n$ which is uniformly bounded), then condition  \hyperref[H2]{(H2)} also holds for $\beta=1/\lambda$ . 
Condition \hyperref[H3]{(H3)} follows from Remark~\ref{rmk:Cconds} and Lemma~\ref{lemma: H3 by uniform branch count}, and is easily verifiable with $\beta_{1,2}=\beta+1$.
Therefore by Theorem~\ref{thm: main theorem},
\begin{equation*}
    \lim_{\alpha\to\infty} \left(1-\dim_H J(\alpha) \right)\alpha^x=
 \begin{cases}
       0 &  \text{ if }  x<\frac{1}{\lambda-1};\\
       \infty & \text{ if } x>\frac1{\lambda-1}.
    \end{cases} 
\end{equation*} 
    \end{example}
\vspace{5mm}

Below in Examples \eqref{ex3}--\eqref{eg:CF} we construct classes of full-branched maps that satisfy our condition for $F$ and $\tau$ in \S\ref{sec:obs}. For every decreasing sequence of positive numbers $(a_n)_n$ such that $\sum_{i=1}^{\infty} a_n=1$, we can associate a partition $(I_n)_n$ of the interval $[0,1]$ such that the length of $I_n$ is equal to $a_n$. The map   $F: \bigcup I_n \to [0,1]$ such that $F$ restricted to each $I_n$ is linear, increasing and $F(I_n)=[0,1]$ is a map that can be coded with a full shift on a countable alphabet.  Clearly this system satisfies (F1)--(F5) and has $b^*=1$. In all examples below, take $\ell(n)=1$.

 \begin{example}\label{ex3}
Fix $r >0$ and $s\in (0, r)$.  Define $(a_n)_n$  and $(b_n)_n$ to be the sequences defined by
 \begin{equation*}
 a_n= C_s n^{-(1+s)} \quad \text{ and } \quad b_n=n^r,
 \end{equation*} 
where $C_s$ is such that $1=C_s\sum_{n=1}^{\infty} a_n\asymp C_s \int_1^{\infty} x^{-(1+s)}= C_s \frac{1}{s}$.  Again  (F1)--(F5) hold, the Lebesgue measure is the measure of maximal dimension and $b^*=1$.
 
Consider the observable defined by $\tau(x) = b_n$ if $x\in I_n$, let $\omega(x)= x^r$, and since $\int_{n^{1/r}}^\infty x^{-1-s}dx\asymp n^{-s/r}$, we verify \hyperref[H2]{(H2)} with $\beta=s/r$ and hence $\beta=s/r\in (0, 1)$ and by Lemma~\ref{lemma: H3 by uniform branch count},  \hyperref[H3]{(H3)} holds. To verify this, we notice that for $b\in(0,1)$,
$$\frac{n^{-(1+s)}}{n^{-b(1+s)}}=n^{-(1+s)(1-b)}=n^{-r\left(\beta+\frac1r\right)(1-b)},$$
so \hyperref[H3]{(H3)} holds with $\beta_{1,2}=\beta+\frac1r$. 
Therefore by Theorem~\ref{thm: main theorem},
\begin{equation*}
    \lim_{\alpha\to\infty} \left(1-\dim_H J(\alpha)\right)\alpha^x=
 \begin{cases}
       0 &  \text{ if }  x<\frac{s}{r-s};\\
       \infty & \text{ if } x>\frac{s}{r-s}.
    \end{cases} 
\end{equation*}  

 \end{example}

\begin{example} \label{ex:exp}
Let $0<c<a$ and $b>0$, set $$a_n = \frac{C}{n^a}, \ c_n = n^c, \ b_n= n^b \text{ such that } a, b, c>0 \text{ and }  \beta= \frac{a-c-1}b\in (0, 1).$$
 Here $C$ is chosen such that $\sum_nCn^{c-a}=1$.  
Consider the partition of $[0,1]$ induced by $(a_n)_n$ and $c_n$, \ie for each $n$ there are $c_n$-number of disjoint intervals of length $a_n$. Define $F$ as in Example \ref{ex3}. Let $\tau :[0,1] \to \R$ defined by $\tau(x)= b_n$ if and only if $x$ belongs to a partition set of length $a_n$. Then $F$ satisfies all the conditions listed in \S\ref{sec:obs} and the measure of maximal dimension is again the Lebesgue measure $\mu$ with $b^*=1$. Let $\omega(x)=x^b$, for all $\eps>0$, $\sum_{n\in\N}n^ce^{-\eps n^b}<\infty$ so \hyperref[H1]{(H1)} holds, and for \hyperref[H2]{(H2)}, $$\mu\{\tau=\omega(n)\}=\frac{C}{n^{a-c}}\asymp\frac{bn^{b-1}}{n^{(1+\beta)b}}.$$
To verify \hyperref[H3]{(H3)}, it suffices to compute that for all $\eta$ close to 1, 
$$\frac{n^{-(a-c)}}{n^cn^{-\eta a}}=n^{-(1-\eta)a}=n^{-b(1-\eta)(\beta+(c+1)/b)},$$
so the relevant $\beta_{1,2}$ here is $\frac ab=\beta+(c+1)/b$ and the conclusion of Theorem~\ref{thm: main theorem} holds, \ie
\begin{equation}
    \lim_{\alpha\to\infty} \left(1-\dim_H J(\alpha)\right)\alpha^x=
 \begin{cases}
       0 &  \text{ if }  x<\frac{a-c-1}{b-a+c+1};\\
       \infty & \text{ if } x>\frac{a-c-1}{b-a+c+1}.
    \end{cases} 
\end{equation}    
 
  \end{example}

\begin{example} \label{ex:exponential}
  Let $\beta\in (0, 1)$.  Continue the construction as in Example \ref{ex:exp}, let $b_n=e^n$, the number of branches with $\tau(x)=b_n$ is given by $c_n=2^n$, and the Lebesgue measure of each partition set with $\tau=e^n$ is $a_n=\frac{K}{2^ne^{\beta n}}$ where $K=(\sum_{n\ge1}e^{-\beta n})^{-1}$. Then $F$ satisfies all the conditions listed in \S\ref{sec:obs} and the measure of maximal dimension is again the Lebesgue measure $\mu$ with $b^*=1$. Clearly, the obvious choice of function here is $\omega(x)=e^x$. Then $\limsup_{x\to\infty}\frac{\omega(x+1)}{\omega(x)}=e$ and $\sum_{n\in\N}2^ne^{-\eps e^n}<\infty$ for all $\eps>0$, so \hyperref[H1]{(H1)} holds. Next, $$\mu(\{\tau\in[e^n,e^{n+1})\})\asymp\frac1{e^{\beta n}}=\frac{ e^{ n}}{e^{(1+\beta)n}},$$
    so \hyperref[H2]{(H2)} is verified. Lastly for \eqref{eqn: perturbations of acim}, for all $b$ close to 1 and $q>0$, 
    $$\frac{\mu(\tau=\omega(n))}{\mu_{q,b}(\tau=\omega(n))}\asymp e^{qe^n}\frac{e^{-\beta n}}{e^{-\beta bn}e^{\log 2(1-b)n}}\le e^{qe^{n+1}}\exp\left(-n(\beta+\log2)(1-b)\right)= e^{q\omega(n+1)}\omega(n)^{-(\beta+\log2)(1-b)},$$
    and similarly the lower bound 
    $$e^{q\omega(n)}\omega(n+1)^{-(\beta+\log2)(1-b)}\lesssim e^{qe^n}\frac{e^{-\beta n}}{e^{-\beta bn}e^{\log 2(1-b)n}}.$$ So \hyperref[H3]{(H3)} holds with $\beta_{1,2}=\beta+\log 2$.    
Therefore, by  Theorem~\ref{thm: main theorem} 
\begin{equation*}
    \lim_{\alpha\to\infty} \left(1-\dim_H J(\alpha)\right)\alpha^x=
\begin{cases}
       0 &  \text{ if }  x<\frac{\beta}{1-\beta};\\
       \infty & \text{ if } x>\frac{\beta}{1-\beta}.
\end{cases} 
\end{equation*} 
\end{example}

\begin{example}[L\"uroth expansions]
\label{eg:Lur}
 Every real number $x \in (0,1)$ has \emph{L\"uroth} expansion of the form:
\begin{equation*}
x=\frac{1}{d_1}+ \frac{1}{d_1(d_1-1)d_2} + \cdots + \frac{1}{d_1(d_1-1) \cdots d_{n-1}(d_{n-1}-1)d_n} + \cdots =[d_1 d_2 \dots]_L
\end{equation*}
with $d_i \geq 2$ a positive integer for every $i \in \N$, see \cite{lu}. Let $F:[0,1] \to [0,1]$ be the map  defined by $F(x)= n(n+1)x-n$ if $x \in [(n+1)^{-1}, n^{-1})$ and $F(0)=0$. If $x=[d_1 d_2 \dots]_L$ then $F(x)=[d_2 d_3 \dots]_L$ (see \cite{dk}). The  Lebesgue measure, that we denote by $\mu$, is preserved by $F$. Then the map $F$ together with the Lebesgue measure satisfies (F1)--(F5) and $b^*=1$.  Let $r \in \N$ be such that  $r>1$ and $\tau:[0,1] \to \R$ be defined by $\tau([d_1 d_2 \dots]_L)= d_1^{r}$. Let $\omega(x)= x^r$. Note that condition \hyperref[H1]{(H1)} is clearly satisfied. Also 
\begin{equation*}
\mu\left(\tau\in [\omega(n),\omega(n+1)) \right)=\mu([n]) \asymp \frac{1}{n(n-1)}\sim\frac1{n^2}.
\end{equation*}
Thus, solving $r(\beta+1)-(r-1)=2$, \hyperref[H2]{(H2)} holds with $\beta  =\frac1r\in(0,1)$. Moreover, Lemma \ref{lemma: H3 by uniform branch count} implies that 
 \hyperref[H3]{(H3)} is satisfied.  Therefore, by  Theorem~\ref{thm: main theorem}, 
\begin{equation*}
    \lim_{\alpha\to\infty} \left(1-\dim_H J(\alpha)\right)\alpha^x=
\begin{cases}
       0 &  \text{ if }  x<\frac{1}{r-1};\\
       \infty & \text{ if } x>\frac{1}{r-1}.
\end{cases} 
\end{equation*} 
\end{example}

\begin{example}[Continued fractions]
Every irrational real number $ x \in (0,1)$  can be written uniquely as a continued fraction of the form
\begin{equation*}
x = \textrm{ } \cfrac{1}{a_1 + \cfrac{1}{a_2 + \cfrac{1}{a_3 + \dots}}} = \textrm{ } [a_1 a_2 a_3 \dots],
\end{equation*}
where $a_i \in \mathbb{N}$ (see \cite[Chapter X]{hw}). The Gauss map,   $F :(0,1] \to (0,1]$, is the interval map defined by 
\begin{equation*}
F(x)= \frac{1}{x} -\Big[ \frac{1}{x} \Big], 
\end{equation*}
where $[x]$ denotes the integer part of the real number $x$. Note that for  $x=[a_1 a_2 a_3 \dots ]$ we have that $F(x)=[a_2 a_3 \dots]$ (see \cite[Section 3.2]{ew}). The map $F$ restricted to the irrational numbers together with the Gauss measure
\begin{equation*}
\mu(A)= \frac{1}{\ln 2} \int_A \frac{1}{1+ x} \, dx,
\end{equation*}
satisfies (F1)--(F5) (condition (F2) is satisfied for the second iterate of $F$, which is sufficient here) and $b^*=1$. We have that
\begin{equation*}
\mu([n]) =  \frac{1}{\ln 2}  \ln \left(1 + \frac{1}{n(n+2)}	\right) = \frac{1}{n^2} - \frac{2}{n^3} + O\left(\frac{1}{n^4} \right).
\end{equation*}
where $[n]=\left\{x=[a_1a_2\dots]\in(0,1]:a_1=n\right\}=\left[\frac1n,\frac1{n+1}\right)$.
Ryll-Nardzewski \cite[Corollary 4]{rn} proved that if $p \geq 1$ then for Lebesgue almost every $x =[a_1, a_2, \dots ] \in (0,1)$,
\begin{equation*}
\lim_{n \to \infty} \left(\frac{a_1^p + a_2^p + \dots + a_n^p }{n}		\right)^{\frac{1}{p}} = \infty. 
\end{equation*}
The Birkhoff spectrum for these type of power mean averages was studied in \cite[Section 6]{IomJor15}, with less precise information on the asymptotic behaviour.  
Let $r>1$ and $\tau:[0,1] \to \R$ be defined by $\tau([a_1 a_2 \dots])= a_1^{r}$. We have that $\int \tau d \mu= \infty$. 
Let $\omega(x)= x^r$. Note that condition \hyperref[H1]{(H1)} is clearly satisfied. We also have the following  estimate, 
\begin{equation*}
\mu\left(\tau\in[\omega(n),\omega(n+1))\right) = \mu([n]) \asymp \frac{1}{n^{2}}. 
\end{equation*}
Thus,  \hyperref[H2]{(H2)} holds again with $\beta  =\frac1r$. Moreover, Lemma \ref{lemma: H3 by uniform branch count} implies that  \hyperref[H3]{(H3)} is satisfied.  Therefore, by  Theorem~\ref{thm: main theorem} 
\begin{equation*}
    \lim_{\alpha\to\infty} \left(1-\dim_H J(\alpha)\right)\alpha^x=
\begin{cases}
       0 &  \text{ if }  x<\frac{1}{r-1};\\
       \infty & \text{ if } x>\frac{1}{r-1}.
\end{cases} 
\end{equation*} 
\label{eg:CF}
\end{example}

\end{document}